\documentclass[12pt,reqno,sort]{elsarticle}

\usepackage{amsfonts,color,amsmath,amssymb,fancyhdr}
\usepackage{amsthm}
\usepackage{graphicx}
\usepackage{tabularx}
\usepackage{soul}
\usepackage{color, xcolor}
\usepackage{indentfirst}

\usepackage{float}
\usepackage{booktabs}
\usepackage{graphicx}
\usepackage{tabularx}

\usepackage[colorlinks,linkcolor=blue, anchorcolor=red,citecolor=green]{hyperref}

\def\Box{\vcenter{\vbox{\hrule\hbox{\vrule
     \vbox to 8.8pt{\hbox to 10pt{}\vfill}\vrule}\hrule}}}

\newtheorem{theorem}{Theorem}[section]
\newtheorem{lemma}[theorem]{Lemma}

\numberwithin{equation}{section}

\newtheorem{hypothesis}[theorem]{Hypothesis}

\definecolor{Purple}{rgb}{0.5,0,0.5}

\begin{document}

\newcommand{\stopthm}{\begin{flushright}
		\(\box \;\;\;\;\;\;\;\;\;\; \)
\end{flushright}}
\newcommand{\symfont}{\fam \mathfam}

\title{Flag-transitive $2$-$(v,k,\lambda)$ designs with $\lambda\ge (r,\lambda)^2$}

\date{}
\author[add1]{Junchi Zhang}\ead{junchizhang@zju.edu.cn}
\author[add2]{Jianbing Lu}\ead{jianbinglu@nudt.edu.cn}
\author[add1]{Meizi Ou\corref{cor1}}\ead{meiziou@zju.edu.cn}\cortext[cor1]{Corresponding author}
\address[add1]{School of Mathematical Sciences, Zhejiang University, Hangzhou 310058,   China}
\address[add2]{Department of Mathematics, National University of Defense Technology, Changsha 410073, China}
\begin{abstract}
    This paper is devoted to the study of $2$-designs with $\lambda\ge (r,\lambda)^2$ admitting a flag-transitive automorphism group $G$. The group $G$ has been shown to be point-primitive of either almost simple or affine type. In this paper, we classify the $2$-designs with $\lambda \geq (r,\lambda)^2>1$ admitting a flag-transitive almost simple automorphism group with socle $\mathrm{PSL}_n(q)$ or $\mathrm{PSU}_n(q)$ for $n \geq 3$.
    \newline
	
	\noindent\text{Keywords:} Flag-transitive; $2$-design; Almost simple group.
	
\end{abstract}	

\maketitle                                         

\section{Introduction}
A $2$-$(v,k,\lambda)$ design is an incidence structure $\mathcal{D} = (\mathcal{P}, \mathcal{B})$ with a set $\mathcal{P}$ of $v$ points and a set $\mathcal{B}$ of $b$ blocks, where each block contains exactly $k$ points, and every pair of distinct points is contained in exactly $\lambda$ blocks. Every point of $\mathcal{D}$ is contained in exactly $r=bk/v$ blocks, which is called the \emph{replication number} of $\mathcal{D}$. We say $\mathcal{D}$ is \emph{non-trivial} if $2<k<v-1$, and symmetric if $v=b$. The design $\mathcal{D}$ is called \emph{simple} if it does not have repeated blocks, and \emph{incomplete} if it is simple  and $b<\binom{v}{k}$. In this paper, we focus on non-trivial incomplete 2-designs.
A \emph{flag} of $\mathcal{D}$ is a point-block pair $(\alpha,B)$, where $\alpha$ is a point and $B$ is a block containing $\alpha$. An automorphism of $\mathcal{D}$ is a permutation of the point set which preserves the block set. The set of all automorphisms of $\mathcal{D}$ under  composition of permutations forms a group, denoted by $\mathrm{Aut}(\mathcal{D})$. For a subgroup $G$ of $\mathrm{Aut}(\mathcal{D})$, $G$ is said to be \emph{point-primitive} if it acts primitively on $\mathcal{P}$, and $G$ is \emph{flag-transitive} if it acts transitively on the set of flags of $\mathcal{D}$. 

The non-trivial $2$-$(v, k, \lambda)$ designs admitting a flag-transitive automorphism group $G$ have been widely studied by several authors. In $1990$, Buekenhout, Delandtsheer, Doyen, Kleidman, Liebeck and Saxl \cite{buekenhout1990linear} classified those $2$-designs with $\lambda=1$ apart from those with a one-dimensional affine automorphism group.  Since then, research attention has focused on the case $\lambda>1$. In $2025$, the classification of flag-transitive $2$-designs with $\lambda=2$ was completed in \cite{regueiro2005biplanes,regueiro2005primitivity,o2007biplanes,o2007classification,o2008biplanes,liang2018flag,liang2016flag2,devillers2021flag,alavi2024classical,liang2025classification}, except for the case where $G\leq\mathrm{A}\Gamma\mathrm{L}_1(q)$. Meanwhile, the flag-transitive $2$-designs with $(r,\lambda)=1$ and $G\nleq\mathrm{A}\Gamma\mathrm{L}_1(q)$ were completely classified in \cite{zhou2015flag2,zhu2015flag,alavi2020flag1,alavi2016symmetric,zhan2016flag,zhan2018non,zhang2019flag2,alavir,alavi2021correction,alavi2022classification,alavi2022finite,biliotti2017flag,biliotti2019nonsymmetric,biliotti2021classification}.  

In his famous book \cite{dembowski1968finite}, Dembowski provided some remarkable conditions, such as $\lambda>(r,\lambda)\left((r,\lambda)-1\right)$, under which any flag-transitive automorphism group of a $2$-design  acts point-primitively. In \cite{zhou2018flag}, Zhan and Zhou proved that if a non-trivial $2$-design  with $\lambda\geq (r,\lambda)^2$ admits a flag-transitive automorphism group $G$, then $G$ is of affine, almost simple or product type. Subsequent studies further explored specific classes of groups, such as sporadic simple groups \cite{zhan2017classification}, and alternating groups \cite{wang2020flag,wang2022alternating}. Notably, Li, Zhang and Zhou \cite{li2024flag} established that for $\lambda \geq (r,\lambda)^2$, no flag-transitive automorphism group of a non-trivial 2-design can be of product type, confining the group structure to affine or almost simple primitivity. Motivated by these developments, we  investigate the case where $G$ is almost simple with socle $\mathrm{PSL}_n(q)$ or $\mathrm{PSU}_n(q)$ for $n \geq 3$. Since the conditions $\lambda=1$ and $(r,\lambda)=1$ are both special cases of  $\lambda \geq (r,\lambda)^2$ and these two cases have been classified, throughout this paper we further assume that $(r,\lambda)>1$. Our main result is as follows:
 
\begin{theorem}\label{main}
    Let $\mathcal{D}$ be a non-trivial $2$-$(v,k,\lambda)$ design with $\lambda\geq (r,\lambda)^2>1$, admitting a flag-transitive automorphism group $G$ with socle $\mathrm{PSL}_n(q)$ or $\mathrm{PSU}_n(q)$, where $n\geq3$ and $q$ is a prime power. Let $H=G_\alpha$ be the stabilizer of a point $\alpha$ of $\mathcal{D}$. \\
    If $\mathrm{Soc}(G)=\mathrm{PSL}_n(q)$, then one of the following holds:
    \begin{enumerate}[(i)]
        \item $H$ is the stabilizer of a $1$-dimensional subspace in $G$;
        \item $H$ is the stabilizer of a $2$-dimensional subspace in $G$, where $n$ is  odd, and $\mathcal{D}$ is non-symmetric;
        \item $H\cong\mathrm{AGL}_1(7)$. Now $G\cong\mathrm{PSL}_3(2):2$, and $\mathcal{D}$ is either a $2$-$(8,4,6)$ or $2$-$(8,4,9)$ design; or $G\cong\mathrm{PSL}_3(2)$ and $\mathcal{D}$ is a $2$-$(8,4,9)$ design.
    \end{enumerate}
     If $\mathrm{Soc}(G)=\mathrm{PSU}_n(q)$, then $n=3$ and $H$ is the stabilizer of a totally singular $1$-dimensional subspace in $G$.
\end{theorem}



This paper is structured as follows. Section \ref{s2} establishes preliminary results on flag-transitive $2$-designs and almost simple groups with socle $\mathrm{PSL}_n(q)$ or $\mathrm{PSU}_n(q)$ for $n \geq 3$. The case where the socle is $\mathrm{PSL}_n(q)$ is resolved in Section \ref{s3}. Finally, Section \ref{s4} addresses the unitary case with socle $\mathrm{PSU}_n(q)$, completing the proof of Theorem \ref{main}.

\section{Preliminaries}\label{s2}

We begin with preliminary results on flag-transitive $2$-designs that will be used throughout this work. For notational convenience, we denote the greatest common divisor of integers $x$ and $y$ by $(x,y)$. Furthermore, we define the reduced parameters: $r^*=r/(r,\lambda)$ and $\lambda^*=\lambda/(r,\lambda)$.

\begin{lemma}\cite[Lemma 2.1]{devillers2021flag}
	\label{designs_property}
	Let $\mathcal{D}$ be a $2$-design with the parameters $(v,b,r,k,\lambda)$. Then the following statements hold:
	\begin{enumerate}[(i)]
		\item $r(k-1)=\lambda(v-1)$. In particular, $r^*$ divides $v-1$;
		\item $bk=vr$;
		\item $b\ge v$ and $r\ge k$;
		\item $\lambda v<r^2$. In particular, if $\lambda\ge (r,\lambda)^2$ then $v<(r^*)^2$.
	\end{enumerate}
\end{lemma}

For convenience, we will  employ freely the results of Lemma \ref{designs_property} hereafter, omitting further references.

\begin{lemma}\cite[2.3.7]{dembowski1968finite} \label{point-primitive}
	Let $\mathcal{D}$ be a $2$-$(v,k,\lambda)$ design with $\lambda \geq (r,\lambda)^2$. If $G \leq \mathrm{Aut}(\mathcal{D})$ acts flag-transitively on $\mathcal{D}$, then $G$ is point-primitive.
\end{lemma}

For a positive integer $n$ and prime number $p$,  let $n_{p}$ and  $n_{p'}$ denote the $p$-part  and  $p'$-part of $n$ respectively, that is, $n_{p}=p^{t}$ where $p^{t}\mid n$ but $p^{t+1}\nmid n$, and $n_{p'}=n/n_{p}$. 

\begin{lemma}\cite[Lemma 2.2, Lemma 2.3]{devillers2021flag}
	\label{deisgns_divisibility} Let $\mathcal{D}$ be a $2$-$(v,k,\lambda)$ design, admitting a flag-transitive  point-primitive automorphism group $G$ with socle $X$, where $X$ is a finite simple group. Let $\alpha$ be a point of $\mathcal{D}$, $H=G_\alpha$ and $H_0=H\cap X.$
	Then  the following statements hold:
	\begin{enumerate}[(i)]
		\item $r$ divides both $|H|$ and $|\mathrm{Out}(X)|\cdot|H_0|$;
		\item $r\mid\lambda s$ for every non-trivial subdegree $s$ of $G$, and so $r^*\mid s$. In particular, if $\lambda \geq (r,\lambda)^2$, then   $v<(r^*)^2\le s^2$ and $v<(v-1,s)^2$;
		\item If a prime $p$ divides $v$, then $(r^*,p)=1$ and $r^*$ divides $|\mathrm{Out}(X)|_{p'}\cdot|H_0|_{p'}$;
        \item $\lambda|G|<|H|^3$.
	\end{enumerate}
\end{lemma}


\begin{lemma}\cite{devillers2021flag}\cite{saxl2002finite}\label{some_subdgrees}
	Let $G$ be an almost simple group with socle $X= \mathrm{PSL}_n(q)$ or $\mathrm{PSU}_n(q)$ with $n \geq 3$. Suppose that $H$ is a maximal subgroup of $G$ not containing $X$. Then the action of $G$ on the cosets of $H$ has subdegrees dividing the numbers $s$ listed in the fourth column of Table \ref{degrees}.
\end{lemma}
\begin{table}[H]
	\caption{Some subdegrees of $X$ on the cosets of $H$ }
	\label{degrees}
	\centering
	\begin{tabular}{lllll}
		\toprule
		Class & $X$ & Type of $H$ & $s$ & conditions\\
		\midrule
		$\mathcal{C}_1$ & $\mathrm{PSL}_n(q)$  & $P_i$  & $q(q^i-1)(q^{n-i}-1)/(q-1)^2$  & $1<i\leq n/2$ \\
        $\mathcal{C}_1$ &  $\mathrm{PSL}_n(q)$ & $P_{i,n-i}$ &	$2q(q^{n-2i}-1)(q^i-1)/(q-1)^2$& $1\leq i\leq n/2$ \\
		$\mathcal{C}_2$ & $\mathrm{PSL}_n(q)$  & $\mathrm{GL}_m(q) \wr \mathrm{S}_2$ &  $4q^{2(m-1)}(q^m-1)^2/(q-1)^2$ &  $n=2m\ge4$ \\
		
		$\mathcal{C}_1$ & $\mathrm{PSU}_n(q)$  & $N_i$  & $\left(q^i-(-1)^i\right)\left(q^{n-i}-(-1)^{n-i}\right)$  & $n\geq 3$ \\
		$\mathcal{C}_2$ & $\mathrm{PSU}_n(q)$  & $\mathrm{GU}_1(q) \wr \mathrm{S}_n$  &  $n(n-1)(n-2) (q+1)^3/6 $  &  $n \geq 4, q \leq 3$ \\
		$\mathcal{C}_2$  & $\mathrm{PSU}_n(q)$  &  $\mathrm{GU}_1(q) \wr \mathrm{S}_n$  &  $n(n-1)(q+1)^2/2$  &  $n \geq 4, q > 3$ \\
		\bottomrule
	\end{tabular}
\end{table}


\begin{lemma}
	\label{lemma|X|<|Out|^2|H_0|....}
	Suppose that  $ \mathcal{D} $  is a  $2$-$ (v, k, \lambda) $ design admitting a flag-transitive automorphism group G with socle $X$, where $X$ is a finite simple group of Lie type in characteristic $p$. Let $\alpha$ be a point of $\mathcal{D}$, $H=G_{\alpha}$ and $H_0=H\cap X$. If $\lambda\ge (r,\lambda)^2$ and $p$ divides $v$, then $|G|<|H|\cdot|H|_{p'}^2$. In particular, we have
    \begin{equation}\label{inequality|X|<|Out(X)^2....|}
    |X|<|\mathrm{Out}(X)|_{p'}^2\cdot|H_0|\cdot|H_0|^2_{p'}.        
    \end{equation}
\end{lemma}
\begin{proof}
	Note that $(r^*,p)=1$ since $r^*$ divides $v-1$. By Lemma \ref{deisgns_divisibility}(i), it follows that  $r^*$ divides $|H|_{p'}$. Then $v<(r^*)^2$ yields $|G|<|H|\cdot|H|^2_{p'}$. Furthermore, as $|G|/|H|=|X|/|H_0|$ and $|H|\leq |\mathrm{Out}(X)|\cdot|H_0|$,  we deduce that $|X|<|\mathrm{Out}(X)|^2_{p'}\cdot|H_0|\cdot|H_0|^2_{p'}$.
\end{proof}

\begin{lemma}
	\label{lem:2.4}
	Suppose that  $ \mathcal{D} $  is a  $2$-$(v, k, \lambda) $ design admitting a flag-transitive automorphism group $G$ with socle $X$, where $X$ is a finite simple group.
 	Let $\alpha$ and $\beta$ be distinct points of $\mathcal{D}$, and suppose $N\le G_{\alpha\beta}$. Then $r^*$ divides $|\mathrm{Out}(X)|\cdot|H_0|/|N|$.
 \end{lemma}
\begin{proof}
	The argument is analogous to \cite[Lemma 2.4]{devillers2021flag}.
\end{proof}

\begin{lemma}
	\label{some_inequalities}\cite[Lemma 4.1]{alavi2015large}
	Let $q$ be a prime power.
	\begin{enumerate}[(i)]
		\item If $a\geq 2$ then
		\[
		(1-q^{-1})^2\leq1 - q^{-1} - q^{-2} < \prod\limits_{j=1}^{a} (1 - q^{-j}) \leq (1 - q^{-1})(1 - q^{-2})
		\]
		\item If $a\geq 3$   then
		\[
		1 < (1 + q^{-1})(1 - q^{-2}) < \prod\limits_{j=1}^{a} \left(1 - (-q)^{-j}\right)\leq(1 + q^{-1})(1 - q^{-2})(1 + q^{-3}).
		\]
	\end{enumerate}
\end{lemma}


\section{Proof of Theorem \ref{main} for $\mathrm{Soc}(G)=\mathrm{PSL}_n(q)$}\label{s3}

Throughout this section, we assume the following hypothesis.
\begin{hypothesis}
	\label{hyp:3.1}
	Let \( \mathcal{D} \) be a non-trivial \( 2\text{-}(v,k,\lambda) \) design with \( \lambda \ge (r,\lambda)^2 > 1 \), admitting a flag-transitive automorphism group \( G \) whose socle is \( X = \mathrm{PSL}_n(q) \), where \( n \ge 3 \) and \( q = p^f \) for some prime \( p \) and positive integer \( f \). Let \( d = (n, q-1) \), and let \( \alpha \) be a point of \( \mathcal{D} \), \( H = G_{\alpha} \) and \( H_0 = H \cap X \).
\end{hypothesis}

By Lemma \ref{point-primitive}, $G$ is point-primitive. Since $X$ is normal in $G$, it is transitive  on the points of $\mathcal{D}$. Thus $v=|G|/|H|=|X|/|H_0|$. 
The primitivity of $G$ further implies that    $H$ is a maximal subgroup of $G$. By Aschbacher's Theorem \cite{aschbacher1984maximal}, $H$ lies in one of the geometric subgroup families $\mathcal{C}_i$ ($1 \leqslant i \leqslant 8$), or in the family $\mathcal{S}$ of almost simple subgroups satisfying 
suitable conditions. When investigating the subgroups in the Aschbacher families, we make frequent use of the information on their structures in \cite[Chap. 4]{kleidman1990subgroup} 

Let \( s \) be a non-trivial subdegree of \( G \), and set \( R = (v - 1, s) \). In~\cite{devillers2021flag}, the authors established the bounds \( v < 2s^2 \) and \( v < 2R^2 \) in their classification of \( 2 \)-\((v,k,2) \) designs admitting a flag-transitive almost simple group \( G \) with socle \( \mathrm{PSL}_n(q) \). These inequalities played a key role in their analysis of maximal geometric subgroups. In our setting, where \( \lambda \geq (r, \lambda)^2 \), we derive significantly stronger constraints: namely, \( v < s^2 \) and \( v < R^2 \) (Lemma~\ref{deisgns_divisibility}(ii)). Hence, all results based on the inequalities \( v < 2s^2 \) and \( v < 2R^2 \) in their work remain valid in our context. The relevant results are summarized in Table \ref{tab:conditions}.

\begin{table}[H]
	\caption{Conclusions  derived by $v<2s^2$ or $v<2R^2$ in \cite{devillers2021flag} }\label{tab:conditions}
	\centering
	\begin{tabular}{llll}
		\toprule
		Class & Type of $H$ &  Conclusion &Reference\\
		\midrule
		$\mathcal{C}_1$ & $P_i$ & $i\leq 2$ &\cite[Lemma 4.2]{devillers2021flag} \\
        $\mathcal{C}_1$  & $P_{i,n-i}$ & $i=1$&\cite[Lemma 4.1]{devillers2021flag}\\
        $\mathcal{C}_1$  & $\mathrm{GL}_i(q)\oplus\mathrm{GL}_{n-i}(q)$ & Impossible&\cite[Lemma 4.1]{devillers2021flag}\\
		$\mathcal{C}_2$  & $\mathrm{GL}_m(q)\wr \mathrm{S}_{t}$ & $t\geq 3$, $m\ge2$&\cite[Lemma 4.3]{devillers2021flag}\\
        $\mathcal{C}_3$ &  $\mathrm{GL}_m(q^{2})$ & $m=2$&\cite[Lemma 4.4]{devillers2021flag}\\
		$\mathcal{C}_5$ & $\mathrm{GL}(n,q^{1/2})$ & $n \le 4$ &\cite[Lemma 4.6]{devillers2021flag}\\
		\bottomrule
	\end{tabular}
\end{table}

\begin{lemma}
	\label{lem:3.2}
	Assume Hypothesis \ref{hyp:3.1}. If   $H \in \mathcal{C}_1$, then  $G \leq  \mathrm{P}\Gamma\mathrm{L}_n(q)$ and $H$  is the stabilizer in  $G$  of an  $i$-dimensional subspace for some $1\leq i\leq n/2$.
\end{lemma}
\begin{proof}
	Suppose that $G \nleq \mathrm{P}\Gamma\mathrm{L}_n(q)$. Then $H$ is of type $P_{i,n-i}$ or $\mathrm{GL}_i(q)\oplus\mathrm{GL}_{n-i}(q)$. From Table \ref{tab:conditions}, $H$ is of type $P_{1,n-1}$. In this case, $v=(q^n-1)(q^{n-1}-1)/(q-1)^2$.
	By  Lemma \ref{deisgns_divisibility}(ii) and Lemma \ref{some_subdgrees}, we have  $r^*\mid s$, where $s=2q(q^{n-2}-1)/(q-1)$ is a subdegree of $G$. Combining this with $r^*\mid(v-1)$, we deduce $r^*\mid \left(v-1,2q(q^{n-2}-1)/(q-1)\right)$. Note that   \[\left(v-1,\frac{2q(q^{n-2}-1)}{q-1}\right)=\frac{2q}{(q-1)^2}\cdot\left(q^{2n-2}-q^{n-1}-q^{n-2}-q+2,(q-1)(q^{n-2}-1)\right).\] Since $(q^{2n-2}-q^{n-1}-q^{n-2}-q+2)-(q-1)^2=(q^n+q^2-q-1)(q^{n-2}-1)$ is divisible by $(q-1)(q^{n-2}-1)$,  we obtain $r^*\mid 2q$. By Lemma \ref{designs_property}(iv), it follows that \[(q^n-1)(q^{n-1}-1)/(q-1)^2=v<(r^*)^2\leq4q^2,\]
    which is impossible for $n\geq3$ and $q\geq2$.
\end{proof}

\begin{lemma}\label{P_i}
    Assume Hypothesis \ref{hyp:3.1}. If  $H \in \mathcal{C}_1$, then  $H$  is the stabilizer in  $G$  of a $1$-dimensional subspace or a $2$-dimensional subspace with $n$ being odd. In the latter case, $\mathcal{D}$ cannot be symmetric.
\end{lemma}
\begin{proof}
    By Lemma \ref{lem:3.2}, $G \leq  \mathrm{P}\Gamma\mathrm{L}_n(q)$ and $H$ is the stabilizer in $G$ of an $i$-dimensional subspace for some $1\leq i\leq n/2$. According to Table \ref{tab:conditions}, we have $i\le2$.
    Suppose $i=2$. Then \[v=\frac{(q^n-1)(q^{n-1}-1)}{(q^2-1)(q-1)}.\] From Table \ref{degrees}, we see that $r^*$ divides
\begin{align*}
     \left(v-1,\frac{q(q+1)(q^{n-2}-1)}{q-1}\right)
    =\frac{q(q^{n-2}-1)}{q^2-1}\cdot \left(\frac{q^n+q^2-q-1}{q-1},(q+1)^2\right).
\end{align*}
If  $n$ is even, then $(q^n+q^2-q-1)/(q-1)\equiv -1 \pmod {q+1}$ and so \[\left((q^n+q^2-q-1)/(q-1),(q+1)^2\right)=1.\] Thus $r^*$ divides  $q(q^{n-2}-1)/(q^2-1)$.
It follows from $v<(r^*)^2$ that $(q+1)(q^n-1)(q^{n-1}-1)<q^2(q^{n-2}-1)^2$, a contradiction. Thus $n$ is odd. Moreover, if $\mathcal{D}$ is symmetric, then the
possibilities for $\mathcal{D}$ can be read off from \cite{dempwolff2001primitive}. But there exists no design that satisfies $\lambda\ge (r,\lambda)^2>1$. This completes the proof.
\end{proof}

\begin{lemma}\label{C_2}
    Assume Hypothesis \ref{hyp:3.1}. Then  $ H \notin \mathcal{C}_2$.
\end{lemma}

\begin{proof}
Suppose that $H$ is a $\mathcal{C}_{2}$-subgroup of type $\mathrm{GL}_m(q)\wr \mathrm{S}_t$,  where $n=mt$. According to Table \ref{tab:conditions}, we have $t\ge3$ and $m\geq 2$. 
 By \cite[proposition 4.2.9]{kleidman1990subgroup},  we have
\[|H_0|=d^{-1}(q-1)^{-1}t!\cdot q^{n(m-1)/2}\prod\limits_{j=1}^m(q^j-1)^t.\]
Since $|X|=d^{-1}\cdot q^{n(n-1)/2}\prod\limits_{j=2}^n\left(q^j-1\right)$ and $|\mathrm{Out}(X)|=2df$, the inequality \eqref{inequality|X|<|Out(X)^2....|} in Lemma \ref{lemma|X|<|Out|^2|H_0|....} yields
  $$
d^{-1}\cdot q^{n(n-1)/2}\prod\limits_{j=2}^n\left(q^j-1\right)<(2df)^2\cdot q^{n(m-1)/2}\cdot \left(d^{-1}(q-1)^{-1}t!\cdot\prod\limits_{j=1}^m(q^j-1)^t\right)^3.
 $$ 
Simplifying this inequality gives
 \begin{equation}\label{eq:3.4 0}
     q^{n(n-2m-1)}<4(q-1)^{-2}f^2(t!)^3\cdot\prod\limits_{j=1}^m(1-q^{-j})^{3t}\cdot\prod\limits_{j=1}^n(1-q^{-j})^{-1}.
 \end{equation}
It follows from Lemma \ref{some_inequalities}(i) that $$\prod\limits_{j=1}^m(1-q^{-j})^{3t}\cdot\prod\limits_{j=1}^n(1-q^{-j})^{-1}<(1-q^{-1})^{3t}(1-q^{-2})^{3t}(1-q^{-1})^{-2}<(1-q^{-1})^2.$$ Combining this with $f^2<2q$, (\ref{eq:3.4 0}) implies that
$q^{n(n-2m-1)}<4f^2(t!)^3q^{-2}<8(t!)^3q^{-1}$, which simplifies to $q^{(t^2-2t)m^2-tm+1}<8(t!)^3$. 
Since $q \ge 2$, by taking logarithm of base $2$, we derive that
$(t^2-2t)m^2-tm+1<3\log_2(t!)+3$. It follows from $m\geq2$ that $4t^2-10t<3\log_2(t!)+2$, which implies $t=3$. Then we have  $3m^2-3m<3\log_23+5$, and so  $m=2$. Moreover, from inequality \eqref{eq:3.4 0}, we deduce that $q=2$. By  Lemma \ref{deisgns_divisibility}(i), $r^*$ divides $2df|H_0|=2592$. Now $v=15554560$, which  contradicts  $v<(r^*)^2$.
\end{proof}

\begin{lemma}\label{C_3}
    Assume Hypothesis \ref{hyp:3.1}. If $H \in \mathcal{C}_3$, then $H\cong \mathrm{AGL}_1(7)$ and \begin{enumerate}[(i)]
        \item $\mathcal{D}$ is a $2$-$(8,4,6)$ or $2$-$(8,4,9)$ design when $G\cong \mathrm{PSL}_3(2):2$;
        \item $\mathcal{D}$ is a $2$-$(8,4,9)$ design when $G\cong \mathrm{PSL}_3(2)$.
    \end{enumerate}
\end{lemma}
\begin{proof}
Suppose that $H$ is a $\mathcal{C}_{3}$ subgroup of type $\mathrm{GL}_m(q^t)$ where  $n=mt$ and $t$ is prime. 
By \cite[Proposition 4.3.6]{kleidman1990subgroup}, we have
\[
	|H_0|=d^{-1}(q-1)^{-1}t\cdot q^{n(m-1)/2}\prod\limits_{j=1}^m(q^{tj}-1).
			\]
 Thus we can compute that
\[v=\frac{|X|}{|H_0|}=t^{-1}q^{n(n-m)/2}\cdot\prod\limits_{j=1}^n (q^{j}-1)\cdot\prod\limits_{j=1}^m (q^{tj}-1)^{-1}.\]
Arguing similarly to the proof of Lemma \ref{C_2}, and applying Lemma \ref{some_inequalities}(i) to the inequality \eqref{inequality|X|<|Out(X)^2....|} in Lemma \ref{lemma|X|<|Out|^2|H_0|....}, we deduce that

\begin{equation}\label{eq:00}
    2^{n^2-2tm^2-tm+1}\le q^{n^2-2tm^2-tm+1}<2^7t^3.
\end{equation}
By taking logarithm of base $2$, we have
\begin{equation}\label{eq:0}
    (t^2-2t)m^2-tm-6<3\log_2t.
\end{equation}
Since $t$ is a prime, we have $t^2-3t-6\leq (t^2-2t)m^2-tm-6<3\log_2t$. Therefore $t<6$.

\textbf{Case 1:} $t=5$. Inequality \eqref{eq:0} implies that $15m^2-5m-6<3\log_25$. Hence $m=1$. From \eqref{eq:00}, we have $q=2$. Now we compute $v=63\cdot 2^{10}.$ By Lemma \ref{deisgns_divisibility}(iii), $r^*$ divides $|\mathrm{Out}(X)|_{p'}\cdot|H_0|_{p'}=155. $ It follows from $v<(r^*)^2$ that $63\cdot 2^{10}<155^2$, which is a contradiction.

\textbf{Case 2:} $t=3$. Inequality \eqref{eq:0} implies that $3m^2-3m-6<3\log_23$, hence $m<3$. 

\textbf{Subcase 2.1:} $m=1$. Then $v=q^3(q^2-1)(q-1)/3$ and $r^*$ divides $2df|H_0|_{p'}=6f(q^2+q+1).$ By $v<(r^*)^2$, we have 
\[q^3(q^2-1)(q-1)<108f^2(q^2+q+1)^2.\]
This inequality holds for $q\in\{2,3,4,5,7,8,9,11,16,27,32\}$. Moreover,
let $R=(6f(q^2+q+1),v-1)$. Then the condition $v<(r^*)^2\le R^2$ implies that $q\in\{2,3,5\}$.

Assume $q=2$. Then $v=8$ and $r^*\mid R=7$. By $v<(r^*)^2$, we have $r^*=7$ and so $r=7(r,\lambda)$. From $r^*(k-1)=\lambda^*(v-1),$ we deduce that $\lambda^*=k-1$. Note that $k<7$ since $\mathcal{D}$ is non-trivial.
\begin{enumerate}[(1)]
    \item If $k=6$, then $\lambda^*=5$. From $vr=bk$ and $\mathcal{D}$ is incomplete, we deduce that $b=28(r,\lambda)/3<\binom{8}{6}=28$, which yields that $(r,\lambda)<3$ and therefore $b\notin\mathbb{Z}$, a contradiction.
    \item If $k=5$, then $\lambda^*=4.$ From $vr=bk$, we deduce that $b=56(r,\lambda)/5$. Since $b\in\mathbb{Z}$, $5\mid(r,\lambda)$. On the other hand, $r$ divides $2df|H_0|=42$ according to Lemma \ref{deisgns_divisibility}(i), thus $(r,\lambda)\mid 6$, a contradiction. 
    \item If $k=4$, then $\lambda^*=3$. It follows from $3(r,\lambda)=\lambda\ge (r,\lambda)^2$ that $(r,\lambda)\le 3$ and so $(r,\lambda)\in\{2,3\}$. For these two cases, since $G$ is flag-transitive, we can verify using Magma \cite{MAGMAbosma1994handbook}  that 
    $H\cong\mathrm{AGL}_1(7)$. And we have $\mathcal{D}$ is a $2$-$(8,4,6)$ or $2$-$(8,4,9)$ design with $G\cong \mathrm{PSL}_3(2):2$; or $\mathcal{D}$ is a $2$-$(8,4,9)$ design with $G\cong \mathrm{PSL}_3(2)$. 
\end{enumerate}
Similarly, considering the integrality of $b$ and the conditions derived above for $q=2$, we can exclude the remaining cases where $k=3$ and $2$.

Assume $q=3$. Then $v=144$, $|H_0|=39$, and $r^*\mid R=13$. By $v<(r^*)^2$, we deduce $r^*=13$ and so $r=13(r,\lambda)$. From $r^*(k-1)=\lambda^*(v-1)$, it follows that $11\lambda^*=k-1$. Combining this $vr=bk$, we deduce that 
\[b=\frac{1872(r,\lambda)}{11\lambda^*+1}\in\mathbb{Z}.\] 
For each  $\lambda^*<r^*=13$, we compute $b$ and find that $b\in\mathbb{Z}$ holds only for $\lambda^*=7$. It follows from  $\lambda v<r^2$  that $7(r,\lambda)\cdot 144<169(r,\lambda)^2$, which implies that $(r,\lambda)\geq6$. On the other hand, Lemma \ref{deisgns_divisibility}(i) shows that $r$ divides $2df|H_0|=78$, so $(r,\lambda)\mid 6$. Hence $(r,\lambda)=6$. 
Then $r=78$ and $b=144$. By Magma \cite{MAGMAbosma1994handbook}, there is no such design satisfying Hypothesis \ref{hyp:3.1}. The case  $q=5$ can be excluded in a similar way, and we omit the details.


\textbf{Subcase 2.2:} $m=2$. Then $v=q^{12}(q^5-1)(q^4-1)(q^2-1)(q-1)/3$ and Lemma \ref{deisgns_divisibility}(iii) shows that $r^*$ divides $2df|H_0|_{p'}=6f(q^2+q+1)(q^6-1)$. It follows from $v<(r^*)^2$ that
\[q^{12}(q^5-1)(q^4-1)(q^2-1)(q-1)<108f^2(q^2+q+1)^2(q^6-1)^2.\]
This inequality holds only for $q=2$. Then we have $r^*$ divides $\left(v-1,6f(q^2+q+1)(q^6-1)\right)=1$, which contradicts  $v<(r^*)^2$.

\textbf{Case 3:} $t=2$. By Table \ref{tab:conditions}, we have $m=2$. Then   $v=q^4(q^3-1)(q-1)/2$ and $|H_0|=2q^2(q^4-1)(q+1)/d$. Note that $2$, $q-1$, and $p$ are all divisors of $v$ and hence coprime to $r^*$. Moreover, from \[v-1=\frac{1}{2}(q+1)(q^7-2q^6+2q^5-3q^4+4q^3-4q^2+4q-4)+1\] we deduce that $(v-1,q+1)=1$ and thus $(r^*,q+1)=1$. It then follows from Lemma \ref{deisgns_divisibility}(i) that $r^*\mid f(q^2+1)$. However, this contradicts the inequality $v<(r^*)^2$. 
\end{proof}

\begin{lemma}\label{PSL:C4}
    Assume Hypothesis \ref{hyp:3.1}. Then  $ H \notin \mathcal{C}_4$.
\end{lemma}
\begin{proof}
Suppose that $H$ is a $\mathcal{C}_{4}$ subgroup of type $\mathrm{GL}_i(q)\otimes \mathrm{GL}_{n/i}(q)$ with $1< i<\sqrt{n}$. In particular $n\ge6$.
    By \cite[Proposition 4.4.10]{kleidman1990subgroup}, we have
    \begin{align*}
        |H_0|=d^{-1}(i,n/i,q-1)\cdot q^{(i^2+n^2/i^2-i-n/i)/2}\prod\limits_{j=2}^i(q^j-1)\cdot\prod\limits_{j=2}^{n/i}(q^j-1).
    \end{align*}
The inequality \eqref{inequality|X|<|Out(X)^2....|} in Lemma \ref{lemma|X|<|Out|^2|H_0|....} yields
\[q^{n^2}<4f^2(q-1)^{-2}q^{2i^2+2n^2/i^2+i+n/i}\prod\limits_{j=1}^i(1-q^{-j})^3\cdot\prod\limits_{j=1}^{n/i}(1-q^{-j})^3\cdot\prod\limits_{j=1}^n(1-q^{-j})^{-1}.\]
By Lemma \ref{some_inequalities}(i), $\prod\limits_{j=1}^i(1-q^{-j})^3\cdot\prod\limits_{j=1}^{n/i}(1-q^{-j})^3\cdot\prod\limits_{j=1}^n(1-q^{-j})^{-1}<(1-q^{-1})^2$. Combining this with $f^2<2q$ we deduce that 
\[q^{n^2}<8q^{2i^2+2n^2/i^2+i+n/i-1}.\]
Let $f(i)=2i^2+2n^2/i^2+i+n/i-1$. This is a decreasing function on the interval $i \in (2,\sqrt{n})$, and hence $f(i)\le f(2)=(n^2+n)/2+9$. Therefore $q^{(n^2-n)/2-9}<8$, which is impossible for $n\ge6$.
\end{proof}

\begin{lemma}\label{gcd}
Assume Hypothesis \ref{hyp:3.1}. Then  $ H \notin \mathcal{C}_5$.
\end{lemma}
\begin{proof}
Suppose that $H$ is a $\mathcal{C}_{5}$-subgroup of type $\mathrm{GL}_n(q_{0})$ with $q=q_{0}^{t}$ and $t$ prime. 
    According to \cite[Proposition 4.5.3]{kleidman1990subgroup}, we have 
 \[|H_0|=d^{-1}\left(n,(q-1)/(q_0-1)\right)\cdot q_0^{n(n-1)/2}\prod\limits_{j=2}^n(q_0^j-1).\]
Then the inequality (\ref{inequality|X|<|Out(X)^2....|}) in Lemma \ref{lemma|X|<|Out|^2|H_0|....} yields
\begin{align}\label{eq:3.6}
    q^{n(n-1)/2}\prod\limits_{j=1}^n(q^j-1)<4f^2q_0^{n(n-1)/2}\cdot(q-1)\left(n,(q-1)/(q_0-1)\right)^3\cdot\prod\limits_{j=2}^n(q_0^j-1)^3.
\end{align}
Since  $(1-q^{-1})^4<(1-q^{-1})^2\leq \prod\limits_{j=1}^n(1-q^{-j})$ by Lemma \ref{some_inequalities}(i), we have  $q^{n(n-1)/2}\prod\limits_{j=1}^n(q^j-1)>(q-1)^4q^{n^2-4}$.
Substituting this into the left-hand side of inequality~\eqref{eq:3.6}, we obtain
\begin{align*}
    (q-1)^4q^{n^2-4}<&4f^2q_0^{n(n-1)/2}\cdot(q-1)\left(n,(q-1)/(q_0-1)\right)^3\cdot\prod\limits_{j=2}^n(q_0^j-1)^3\\
    <&8q\cdot q_0^{n(n-1)/2}\cdot 8q^{-3}_0\cdot (q-1)^4\cdot q_0^{3(n^2+n-2)/2},
\end{align*}
 which yields $q^{(n^2-5)t-2n^2-n+6}_0<64$. Taking logarithms, we deduce
that
\begin{equation}\label{eq:3.7}
    (n^2-5)t-2n^2-n<0.
\end{equation}
We consider the following cases separately:

\textbf{Case 1:} $t\ge5$. From \eqref{eq:3.7}, we deduce that $n=3$. From  \eqref{eq:3.6} we see that 
\[q_0^{15}(q_0^{10}-1)(q_0^{15}-1)\le q^3(q^2-1)(q^3-1)<4f^2q_0^3\cdot 27(q_0^2-1)^3(q_0^3-1)^3\]
which is impossible.

\textbf{Case 2:} $t=3$. From \eqref{eq:3.7} we deduce that $n<5$. For $n=4$,  \eqref{eq:3.6} does not hold. Thus $n=3$. Since $q_0=p^{f/3}$, we have $3\mid f$. From \eqref{eq:3.6}, we see that 
\[q_0^6(q_0^6-1)(q_0^9-1)<108f^2(q_0^2-1)^3(q_0^3-1)^3<108f^2(q_0^6-1)(q_0^9-1).\]
Hence $q_0^6<108f^2$, which yields $q_0=2$ or $3$. Therefore $\left(3,(q-1)/(q_0-1)\right)=1$ and we use \eqref{eq:3.6} again to obtain a contradiction.

\textbf{Case 3:} $t=2$.  From Table \ref{tab:conditions},  we deduce that $n\le 4$. Moreover,
by the proof of \cite[Lemma 4.6]{devillers2021flag}, there exists a point $\beta$ of $\mathcal{D}$ such that $X_{\alpha\beta}$ contains a subgroup isomorphic to $\mathrm{SL}_{n-2}(q_0)$, and hence so does $G_{\alpha\beta}$. Applying Lemma \ref{lem:2.4}, it follows that $r^*$ divides $2df|H_0|/|\mathrm{SL}_{n-2}(q_0)|$. This implies that $r^*$ divides $2f\cdot(n,q_0+1)\cdot q_0^{2n-3}(q_0^n-1)(q_0^{n-1}-1)$. Since $2\mid v$ and $p\mid v$, we then conclude that $r^*$ divides
\begin{equation}\label{eq:3.6'}
     f\cdot(n,q_0+1)(q_0^n-1)(q_0^{n-1}-1).
\end{equation}

\textbf{Subcase 3.1:} $n=4$. Then $v=q_0^6(q_0^4+1)(q_0^3+1)(q_0^2+1)/(4,q_0+1)$. It follows from $q_0^2+1\mid v$, $q_0+1\mid v$ and  \eqref{eq:3.6'} that $r^*$ divides $f\cdot(q_0^2-1)(q_0^3-1).$ From $v<(r^*)^2,$ we obtain
$(q_0^4+1)(q_0^3+1)(q_0^2+1)<8q_0^6$, a contradiction.

\textbf{Subcase 3.2:} $n=3$. Then $v=q_0^3(q_0^3+1)(q_0^2+1)/(3,q_0+1)$. By \eqref{eq:3.6'}, $r^*$ divides $f\cdot (3,q_0+1)(q_0^3-1)(q_0^2-1)$.
We first consider the case $(3, q_0 + 1) = 3$ and the other case is analogous. 
In this case, $r^*$ divides $3f\cdot\left(3(v-1),(q_0^3-1)(q_0^2-1)\right)$. Now observe that   
\[\left(3(v-1),(q_0^3-1)(q_0^2-1)\right)=(2q_0^2+4q_0+5,q_0^3+2q_0^2+2q_0+1)\]
divides $(2q_0^2+4q_0+5,q_0-2)$ if $q_0\ne2$. Then   we conclude that $v<(r^*)^2$ fails, thus $q_0=2$. Now $v=120$ and $r^*$ divides $\left(v-1,3f\cdot(q_0^3-1)(q_0^2-1)\right)=(119,126)=1$, which  contradicts   $v<(r^*)^2$. This completes the proof.
\end{proof}

\begin{lemma}
    Assume Hypothesis \ref{hyp:3.1}. Then $H \notin \mathcal{C}_6$.
\end{lemma}
\begin{proof}
    Suppose that $H$ is a $\mathcal{C}_6$ subgroup of type  $t^{2m}.\mathrm{Sp}_{2m}(t)$, where  $n=t^m$  for some prime  $ t \neq p $  and positive integer  $ m $. Moreover,  $f$  is odd and minimal such that  $ t(2, t) $  divides  $ q-1 = p^f - 1$. From \cite[Propositions 4.6.5 and 4.6.6]{kleidman1990subgroup},  we have \[|H_0|\le t^{2m}|\mathrm{Sp}_{2m}(t)|\le  t^{2m^2+3m}<q^{2m^2+3m}.\]
    As $f^2<2q$, the inequality (\ref{inequality|X|<|Out(X)^2....|}) in Lemma \ref{lemma|X|<|Out|^2|H_0|....} yields that $|X|<4d^2f^2|H_0|^3<8q^{6m^2+9m+3}$.
    Combining this with the fact that $|X|>q^{n^2-2}$, we obtain $q^{t^{2m}-6m^2-9m-5}<8$ and so 
    \begin{equation}\label{eq:3}
         t^{2m}-6m^2-9m<8.
    \end{equation}
     As $t\ge2$, we deduce that $m\le3$.  We   provide  details only for  $m=1$, and the other cases are similar.
For $m=1$, it follows from (\ref{eq:3}) that $t^2\le22$, and hence $t=n=3$. Then $|H_0| \leq 3^{2} |\mathrm{Sp}_{2}(3)|=2^3 \cdot 3^3$. Since $t(2,t)$ divides $q-1$, we obtain $d=(n,q-1)=3$. Using the inequality $ q^{n^2-2} < |X| < 4(df)^2 |H_0|^3 $ we obtain $q^7 < f^2 \cdot 6^{11}$.
Combining this inequality with that $f$  is the minimal odd integer such that $ t(2,t) = 3 $  divides  $ p^f - 1 $, we have $q=7$ or $13$. According to \cite[Proposition 4.6.5]{kleidman1990subgroup}, we have $H_0\cong 3^2.\mathrm{Q}_8$. By Lemma \ref{deisgns_divisibility}(i), $r^*$ divides $2df|H_0|=432$. Hence $r^*$ divides $R=(432,v-1)$. Direct calculations show that $R=3$ for both $q=7$ and $q=13$. But then $R^2<v$, a contradiction.
\end{proof}

\begin{lemma}
    Assume Hypothesis \ref{hyp:3.1}. Then  $ H \notin \mathcal{C}_7$.
\end{lemma}
\begin{proof}
Suppose that $H$ is a $\mathcal{C}_7$ subgroup of type $\mathrm{GL}_m(q)\wr \mathrm{S}_t$, where $m\geq3$, $t\geq2$ and $n=m^t$. From \cite[Proposition 4.7.3]{kleidman1990subgroup}, we deduce that $|H_0|<q^{t(m^2-1)}\cdot (t!)$. Then the inequality (\ref{inequality|X|<|Out(X)^2....|}) in Lemma \ref{lemma|X|<|Out|^2|H_0|....} yields $|X|<8q^{3t(m^2-1)+3}(t!)^3$.  Combining this with the inequality $|X|>q^{n^2-2}$, we obtain 
\[2^{3^{2t}-24t-5}\leq q^{m^{2t}-3t(m^2-1)-5} \leq 8(t!)^3.\] Thus we derive that
$3^{2t}-24t-8<3t\log_2t$, which is impossible for $t\ge2$.
\end{proof}

\begin{lemma}\label{our method}
    Assume Hypothesis \ref{hyp:3.1}. Then $ H \notin \mathcal{C}_8 $
\end{lemma}
\begin{proof}
Let $H$ be a $\mathcal{C}_{8}$ subgroup. Then three cases arise.

 \textbf{Case 1:} Suppose that $H$ is of type $\mathrm{Sp}_n(q)$, where $n$ is even and $n\geq4$. By \cite[Proposition 4.8.3]{kleidman1990subgroup}, we compute
 $$
\left|H_0\right| = d^{-1}(n/2,q-1)q^{n^2/4} (q^n - 1) (q^{n-2} - 1) \cdots (q^2 - 1) ,
$$ 
 and so
 $$
v = \frac{|X|}{|H_0|} = \frac{q^{(n^2 - 2n)/4} (q^{n-1} - 1) (q^{n-3} - 1) \cdots (q^3 - 1)}{(n/2,q-1)}.
$$ 
Observing that $p\mid v$ and $r^*$ divides $2df|H_0|$, then $r^*_p=1$ and $r^*$ divides
\begin{equation}\label{eq:divides}
    2f\cdot(q^n - 1) (q^{n-2} - 1) \cdots (q^2 - 1)\cdot(n/2,q-1).
\end{equation}
We consider the following cases:

\textbf{Subcase 1.1:} $n=4$. Then $v=q^2(q^3-1)/(2,q-1)$ and $r^*$ divides $2f\cdot(q^4-1)(q^2-1)(2,q-1)$ from \eqref{eq:divides}. If $q$ is even, by calculating $\left(v-1,2f\cdot(q^4-1)(q^2-1)\right)$, we deduce  $r^*\mid  f\cdot(q^5-q^2-1,q^3+q-1)$ and so $r^*\mid 9f$. According to $v<(r^*)^2$, we have $q=2$. Then a Magma \cite{MAGMAbosma1994handbook} computation shows that the subdegrees of $G$ are 12 and 15, so by Lemma \ref{deisgns_divisibility}(ii) $r^*$ divides $3$, contradicting $v<(r^*)^2$. In the case where $q$ is odd, we have $(q-1)/2$ divides $v$. Therefore, $r^*$ divides $16f(q+1)^2(q^2+1)$. By calculating $a:=\left(2(v-1),16f(q+1)^2(q^2+1)\right)$, we deduce  that $r^*$ divides $16f(q-11)$ if $q\ne 11$. From $v<(r^*)^2$, it follows that $q\in\{3,5\}$. For $q\in\{3,5,11\}$, we recompute $a$ for each $q$ and find  $v>a^2$, a contradiction.

\textbf{Subcase 1.2:} $n\ge6.$ By the proof of \cite[Lemma 4.9]{devillers2021flag}, there exists a point $\beta$ of $\mathcal{D}$ such that $X_{\alpha\beta}$ contains a subgroup isomorphic to $\mathrm{Sp}_{n-4}(q)$, and hence so does $G_{\alpha\beta}$. By Lemma \ref{lem:2.4}, $r^*$ divides $2df|H_0|/|\mathrm{Sp}_{n-4}(q)|$, that is,
\begin{equation}\label{eq:|}
    2(n/2,q-1)f\cdot q^{2n-4}(q^n-1)(q^{n-2}-1).
\end{equation}
Since $2$, $p$ and $q-1$ are all divisors of $v$, then $r^* \mid v-1$ and (\ref{eq:|}) shows that $r^*$ divides
\[f\cdot (q^n-1)(q^{n-2}-1)/(q-1)^2.\]
By $v<(r^*)^2$, we obtain $n\le8$. Using an analogous argument to \textbf{Subcase 1.1}, we find that both $n=8$ and $n=6$ yield contradictions to  $v<(r^*)^2$.

\textbf{Case 2:}    Suppose that $H$ is of type $\mathrm{O}_n^\epsilon(q)$, where $\epsilon\in\{\circ,\pm\}$ and $q$ is odd. Then by \cite[Proposition 4.8.4]{kleidman1990subgroup}, $H_0\cong \mathrm{PSO}^{\epsilon}_n(q).(n,2)$. By the proof of \cite[Lemma 4.10]{devillers2021flag}, there exists a point $\beta$ of $\mathcal{D}$ such that $X_{\alpha\beta}$ contains a subgroup isomorphic to $\mathrm{SO}^{\epsilon}_{n-2}(q)$, and hence so does $G_{\alpha\beta}$. It follows from Lemma \ref{lem:2.4} that $r^*$ divides $2df|H_0|/|\mathrm{SO}^{\epsilon}_{n-2}(q)|$. Combined with $v < (r^*)^2$, this condition eliminates all possibilities in this case, 
as demonstrated in the analogous proof of \cite[Lemma 4.10]{devillers2021flag}.


\textbf{Case 3:}    Suppose that $H$ is of type $\mathrm{U}_n(q_0)$, where $q=q_0^2$. Then by \cite[Proposition 4.8.5]{kleidman1990subgroup}, we obtain   
    \[|H_0|=d^{-1}c\cdot q_0^{n(n-1)/2}\prod\limits_{j=2}^n(q_0^j-(-1)^j),\] 	where $c=(q-1)/\operatorname{lcm}(q_{0}+1,(q-1)/d)=(n,q_{0}-1)$. Hence
     \[v=\frac{|X|}{|H_0|}=\frac{1}{c}q_0^{n(n-1)/2}\prod\limits_{j=2}^n(q_0^j+(-1)^j),\]
which implies $p\mid v$ and $v$ is even. It follows from Lemma \ref{deisgns_divisibility}(i) that $r^*$ divides
\[(n,q_0-1)f\cdot\prod\limits_{j=2}^n(q_0^j-(-1)^j).\] 
\indent \textbf{Subcase 3.1:} If $n=3$, then
 $v=q_0^3(q_0^3-1)(q_0^2+1)/(3,q_0-1)$ and $r^*\mid (3,q_0-1)f(q_0^3+1)(q_0^2-1)$. Following an analogous argument to  \textbf{Subcase 3.2} of Lemma \ref{gcd}, we can rule out this case.

\textbf{Subcase 3.2:} If $n \ge4$, then by the proof of \cite[Lemma 4.11]{devillers2021flag},  there exists a point $\beta$ of $\mathcal{D}$ such that $X_{\alpha\beta}$ contains a subgroup isomorphic to $\mathrm{SU}_{n-2}(q)$, and hence so does $G_{\alpha\beta}$. It follows from Lemma \ref{lem:2.4} that $r^*$ divides $2df|H_0|/|\mathrm{SU}_{n-2}(q)|$, and so
$r^*\mid cf(q_0^n-(-1)^n)(q_0^{n-1}-(-1)^{n-1})$.
This case can be excluded by an argument analogous to that in \textbf{Case 1}, and we omit the details.
\end{proof}


\begin{lemma}\label{S}
    Assume Hypothesis \ref{hyp:3.1}. Then  $ H \notin \mathcal{S}$.
\end{lemma}
\begin{proof}
Suppose that $H$ is an $\mathcal{S}$-subgroup. Then by Lemma \ref{deisgns_divisibility}(iv), we obtain
\[
q^{n^2-2} < |X| \leq |G| < |H|^3.
\]
Moreover, by \cite[Theorem 4.1]{liebeck1985orders}, we conclude   $q^{n^2-2}<|H|^3< q^{9n}$, yielding  $n\leq9$. Further, it follows from \cite[Corollary 4.3]{liebeck1985orders} that either $n=x(x-1)/2$ for some integer $x$ or $|H|<q^{2n+4}$. Both cases lead to $n\le7$. For $n \leq 7$, the possibilities for $H_0$ can be read off from the tables in \cite[Chapter 8]{bray2013maximal}.  It follows from Lemma \ref{lemma|X|<|Out|^2|H_0|....}  that $q^{n^2-2} < |X| < 4d^2f^2\cdot |H_0|^3$. Since $d\leq n$, we obtain 
\begin{equation}\label{s_subgroup}
   q< (4f^2n^2|H_0|^3)^{1/(n^2-2)}.
\end{equation}
Using this inequality and the conditions on $q$,  we list all pairs $(X,H_0)$ that satisfy the inequality in Table \ref{tab:S_subgroup0}. The conditions in the fourth column can be obtained from \cite{bray2013maximal}. For Line 8,  Lemma \ref{lemma|X|<|Out|^2|H_0|....} gives $q^{34}< |\mathrm{PSL}_6(q)| < 4d^2f^2\cdot |\mathrm{PSL}_3(q)|^3<144q^{26}$, which is impossible. For each of the remaining cases listed in Lines 1-7 of Table \ref{tab:S_subgroup0}, we verify whether the inequality $v<\left(v-1,|\mathrm{Out}(X)|\cdot|H_0|\right)^2$ is valid.   Only the case   $q=2$  in Line 4 satisfies this inequality.   If $X=\mathrm{PSL}_4(2)$ and $H_0\cong \mathrm{A}_7$, then $v=8$ and $r^*$ divides $2df|H_0|=5040$. Therefore, $r^*\mid (7,5040)=7$. From $v<(r^*)^2$, we deduce that $r^*=7$. Following calculations similar to those in Lemma \ref{C_3}, we confirm that such $2$-designs do not exist.
\end{proof}

\begin{table}[H]
\caption{\text{The pairs $(X,H_0)$}}
\label{tab:S_subgroup0}
\centering
\begin{tabular}{lllll}
\toprule
Line & $X$ & $H_0$ & \text{Conditions on } $q$ & \text{Restriction on } $q$\\
\midrule
        1  & $\mathrm{PSL}_3(q)$  & $\mathrm{PSL}_2(7)$  & $q=p\equiv 1,2,4 \pmod  7$, $q\neq2$  & $q=11$ \\
        2  &             & $\mathrm{A}_6$    &  $q=p\equiv 1,4 \pmod{15}$    &  $q=19$ \\
        3  &             & $\mathrm{A}_6 $  & $q=p^2$, $p\equiv 2,3\pmod5$  & $q=4$ \\
       
\midrule
         4& $\mathrm{PSL}_4(q)$  & $\mathrm{A}_7$ &$q=p\equiv 1,2,4 \pmod 7$ & $q=2$\\
         5&   & $\mathrm{PSU}_4(2)$ & $q=p\equiv 1\pmod6$ & $q=7$\\
         
\midrule

         6&  $\mathrm{PSL}_5(q)$  & $\mathrm{M}_{11}$    & $q=3$ &  $q=3$ \\
      
\midrule
      
      7  &$\mathrm{PSL}_6(q)$& $\mathrm{M}_{11}$    & $q=3$ &  $q=3$  \\
      8  && $\mathrm{PSL}_3(q)$  &$q$ odd & \\

\bottomrule
\end{tabular}
\end{table}

\section{Proof of Theorem \ref{main} for $\mathrm{soc}(G)=\mathrm{PSU}_n(q)$}\label{s4}
In this section, we assume the following hypothesis.

\begin{hypothesis}
	\label{Hypothesis}
	Let \( \mathcal{D} \) be a non-trivial \( 2\text{-}(v,k,\lambda) \) design with \( \lambda \ge (r,\lambda)^2 > 1 \), admitting a flag-transitive automorphism group \( G \) whose socle is \( X = \mathrm{PSU}_n(q) \), where \( n \ge 3 \) and \( q = p^f \) for some prime \( p \) and positive integer \( f \), and $(n,q) \neq (3,2)$. Let \( d = (n, q+1) \), and let \( \alpha \) be a point of \( \mathcal{D} \), with \( H = G_{\alpha} \) and \( H_0 = H \cap X \).
\end{hypothesis}

    Following the same approach as in the case \( X = \mathrm{PSL}_n(q) \), we first recall some useful results from previous literature concerning the bounds \( v < s^2 \) or \( v < R^2 \),  which remain applicable here.  We summarize these results in Table \ref{tab:PSU_conditions}. 

\begin{table}[H]
	\caption{Conclusions  derived by $v<s^2$ or $v<R^2$ }
    \label{tab:PSU_conditions}
	\centering
	\begin{tabular}{llll}
		\toprule
		Class & Type of $H$ &  Conclusion & References\\
		\midrule
		$\mathcal{C}_1$ & $P_i$ & $n=3$ & \cite[Section 6(1)]{saxl2002finite} \\
            $\mathcal{C}_1$ & $N_i$ & $i=1$ & \cite[Proposition 4.3]{alavir} \\
            $\mathcal{C}_2$ & $\mathrm{GU}_m(q)\wr \mathrm{S}_{t}$ & $m=1$ & \cite[Proposition 4.3]{alavir} \\
            $\mathcal{C}_2$ & $\mathrm{GL}_{n/2}(q^2).2$ & $n=4$ & \cite[Proposition 4.3]{alavir} \\
		\bottomrule
	\end{tabular}
\end{table}

In the case \( X = \mathrm{PSL}_n(q) \), we frequently utilize the inequality \eqref{inequality|X|<|Out(X)^2....|} 
from Lemma \ref{lemma|X|<|Out|^2|H_0|....} for parameter estimation. 
For \( X = \mathrm{PSU}_n(q) \), we apply the established framework from the classification 
of flag-transitive \(2\)-designs with \((r,\lambda)=1\) admitting an almost simple group \(G\) 
with socle \(X = \mathrm{PSU}_n(q)\) \cite{alavir}. In particular, \cite[Lemma 3.7]{alavir} 
yields precisely the same inequality as (\ref{inequality|X|<|Out(X)^2....|}), 
and their methodology extends naturally to the remaining cases under our conditions. 
This allows us to directly utilize the classification results from \cite{alavir} 
to eliminate certain types of $H$. Additionally, we leverage results from \cite{alavi2024classical}, which established an alternative bound for flag-transitive \(2\)-\((v,k,2)\) designs with socle \( X \ne \mathrm{PSL}_n(q) \). Specifically, \cite[Lemma 2.11]{alavi2024classical} shows that $|X| < 2|\mathrm{Out}(X)|^2\cdot |H_0| \cdot |H_0|_{p'}^2$. Our general bounds (\ref{inequality|X|<|Out(X)^2....|}) imply those used in \cite{alavi2024classical}, allowing us to incorporate their classification results. In conclusion, we summarize these excluded $H$-types in Table \ref{tab:excluded-H-types}, and will omit their consideration in subsequent arguments.


\begin{table}[H]
	\caption{Excluded $H$-types derived in \cite{alavir,alavi2024classical}}
    \label{tab:excluded-H-types}
	\centering
	\begin{tabular}{lll}
		\toprule
		Class & Type of $H$& References\\
		\midrule
		$\mathcal{C}_3$ & $\mathrm{GU}_m(q^t)$  & \cite[Proposition 4.3]{alavir} \\
          $\mathcal{C}_4$ & $\mathrm{GU}_i(q) \otimes \mathrm{GU}_{n/i}(q)$  & \cite[Corollary 2.12]{alavi2024classical} \\
            $\mathcal{C}_5$ & $\mathrm{GU}_n(q_0)$  & \cite[Proposition 4.3]{alavir} \\
            $\mathcal{C}_6$ & $t^{2m}.\mathrm{Sp}_{2m}(t)$  & \cite[Proposition 4.1]{alavir} \\
             $\mathcal{C}_7$ & $\mathrm{GU}_m(q) \wr \mathrm{S}_t$  & \cite[Corollary 2.12]{alavi2024classical} \\
		\bottomrule
	\end{tabular}
\end{table}

\begin{lemma}
    Assume Hypothesis \ref{Hypothesis}. If  $H \in \mathcal{C}_1$, then $n=3$ and $H$ is the stabilizer of a totally singular $1$-dimensional subspace.
\end{lemma}
    \begin{proof}
        If $H$ is a $\mathcal{C}_1$ subgroup, then it is either a parabolic subgroup $P_i$ or the stabilizer $N_i$ of a non-singular subspace. Suppose that $H$ is of type $P_i$ for some $i \leq n/2$. From Table \ref{tab:PSU_conditions}, $H$ is of type $P_1$ with $n=3$.
Suppose that $H$ is isomorphic to $N_i$ with $i<n/2$. From Table \ref{tab:PSU_conditions}, we have $i=1$. 
In this case, let  $R=(v-1,s)$, where $s=\left(q+1\right)\left(q^{n-1}-(-1)^{n-1}\right)$ as in Lemma \ref{some_subdgrees}.
       If $n$ is odd, then by \cite[Proposition 4.1.4]{kleidman1990subgroup}, we compute that $v=|X|/|H_0|=q^{n-1}(q^n+1)/(q+1)$ and $R= (q^{n-1}-1)/(q+1)$. It follows from $v< R^2$ that \[q^{2n-1}<q^{n-1}(q^n+1)<(q^{n-1}-1)^2/(q+1)<q^{2n-3},\] a contradiction. Now suppose that $n$ is even. Then $v=q^{n-1}(q^n-1)/(q+1) $ and $R= (q^{n-1}+1)/(q+1)$. So $v< R^2$ implies that
        $q^{n-1}(q^n-1)<(q^{n-1}+1)^2/(q+1)$, which is impossible for $n\geq3$ and $q\geq2$.
        \end{proof}

  \begin{lemma}
      Assume Hypothesis \ref{Hypothesis}. Then  $H \not \in \mathcal{C}_2$. 
  \end{lemma}     
  \begin{proof}
  Let $H $ be a $ \mathcal{C}_2$ subgroup. Then two cases arise.

\textbf{Case 1:} Suppose that $H$ is a $\mathcal{C}_2$ subgroup of type $\mathrm{GU}_m(q)\wr \mathrm{S}_t$, where $n=mt$. 
From Table \ref{tab:PSU_conditions}, $m=1$.
 First assume $n=3$. 
Then by \cite[Proposition 4.2.9]{kleidman1990subgroup}, we derive that $|H_0|=6d^{-1}(q+1)^2$ and $v=q^3(q^2-q+1)(q-1)/6$. By Lemma \ref{deisgns_divisibility}(iii), we have  $r$ divides $12f(q+1)^2$. Hence $4v\le(r,\lambda)^2v \leq \lambda v<r^2 \leq 144f^2(q+1)^4$, that is
        $$
        q^3(q^2-q+1)(q-1) < 216f^2(q+1)^4.
        $$
        This inequality holds  when $p=2$ and $2\leq f\leq6$; when $p=3$ and $f\leq3$; when $p=5$ and $f\leq2$; or when $p\in\{7,11,13,17\}$ and $f=1$.
For each  case, we compute $v$ and evaluate the expression $\left(12f(q+1)^2,v-1\right)$. In every instance, $v<(r^*)^2$ leads to a contradiction.
Hence $n \geq 4$. 
From Lemma \ref{some_subdgrees}, $r^*$ divides $s_1$, where $s_1=n(n-1)(q+1)^2/2$ if $q>3$, and $s_1=n(n-1)(n-2)(q+1)^3/6$ if $q\leq3$. We only give the details for the case where $q>3$ and the other case is similar. For $q>3$, we have $
        v=(q+1)^{-n}\cdot|\mathrm{GU}_n(q)|/(n!). 
        $ It follows from \cite[Lemma 4.2]{alavi2015large} that $|\mathrm{GU}_n(q)|\geq q^{n^2}$. As $(q+1)^n\le q^{2n}$ we obtain $q^{n^2-2n}/(n!)<v<(r^*)^2\leq s_1^2$. Hence
$4q^{n^2-2n}<(n!)\cdot n^2(n-1)^2(q+1)^4$, which implies $n=4$, and $q=4$ or $5$. Now $v=q^6(q-1)^2(q^2+1)(q^2-q+1)/24$ and $s_1=6(q+1)^2$. Then we deduce $v>(v-1,s_1)^2$, a contradiction. 

\textbf{Case 2:} Suppose that $H$ is of type $\mathrm{GL}_{n/2}(q^2).2$. According to Table \ref{tab:PSU_conditions}, this implies that $n=4$. The remainder of the proof proceeds similarly to the argument  in \cite[Proposition 4.3]{alavir}.
\end{proof}

\begin{lemma}
    Assume Hypothesis \ref{Hypothesis}. Then $H \not \in \mathcal{C}_5$. 
\end{lemma}
\begin{proof}
Let $H$ be a $\mathcal{C}_5$ subgroup. From Table \ref{tab:excluded-H-types} we have the remaining two cases.

\textbf{Case 1:} Suppose that $H$ is of type $\mathrm{O}_n^\epsilon(q)$ with $q$ odd and $\epsilon \in \{\circ,\pm\}$. According to \cite[Proposition 4.5.5]{kleidman1990subgroup}, we have $H_0 \cong \mathrm{PSO}_n^\epsilon(q).(2,n)$.

\textbf{Subcase 1.1:} Assume that $n$ is odd and $n\neq3$. In this case, we have $|H_0|=|\mathrm{PSO}_n^\circ(q)|$ and 
\begin{equation}\label{Otype}
    v=d^{-1}q^{(n^2-1)/4} (q^3+1)(q^5+1)\cdots(q^n+1).
\end{equation}
As $q+1<2q$, we obtain that $v>\frac{1}{2}q^{n^2/2+n/2-2}$. By Lemma \ref{deisgns_divisibility}, $r^*$ divides
\begin{equation}\label{O_N.r^*divides}
    2df\cdot q^{(n-1)^2/4} (q^2-1)(q^4-1)\cdots(q^{n-1}-1).
\end{equation}
Observe that $v$ is even and we select some factors of $v$ as follows:
\begin{align*}
    & 2,q,q+1,q^3+1,\cdots,q^{(n-1)/2}+1  \text{, \quad when } (n-1)/2 \text{ is odd.}\\
    &  2,q,q+1,q^3+1,\cdots,q^{(n-3)/2}+1 \text{, \quad when } (n-1)/2 \text{ is even.}
\end{align*}
Since $r^* \mid (v-1)$, we deduce that $r^*$ is coprime to each of the above factors of $v$. Consequently, simplifying (\ref{O_N.r^*divides}) yields that $r^*$ divides
\begin{align*}
   & f \cdot \frac{(q^2-1)(q^4-1)\cdots(q^{n-1}-1)}{(q+1)(q^3+1)\cdots(q^{(n-1)/2}+1)} \leq f\cdot q^{\frac{3n^2}{16}-\frac{n}{8}-\frac{5}{16}} \text{, \quad when } (n-1)/2 \text{ is odd};\\
   & f \cdot \frac{(q^2-1)(q^4-1)\cdots(q^{n-1}-1)}{(q+1)(q^3+1)\cdots(q^{(n-3)/2}+1)}\leq f\cdot q^{\frac{3n^2}{16}+\frac{n}{8}-\frac{5}{16}}   \text{, \quad when } (n-1)/2 \text{ is even}.
\end{align*}
Combining these two cases with $v<(r^*)^2$, we obtain
$$
v<(r^*)^2 \le f^2 q^{\frac{3n^2}{8}+\frac{n}{4}-\frac{5}{8}}.
$$
Therefore, by $v>\frac{1}{2}q^{n^2/2+n/2-2}$ and $2f^2<q^2$, we derive that $n^2+2n-27<0$. However, this inequality has no valid $n$.

\textbf{Subcase 1.2:} $n=3$. Then $v=d^{-1}q^2(q^3+1)$ where $d=(q+1,3)$. It follows from \eqref{O_N.r^*divides} that $r^*$ divides $2df \cdot q(q^2-1)$. If $d=1$, then $v$ is even and $p \mid v$. Then $r^*$ divides $f(v-1,q^2-1)=f$, and $v>(r^*)^2$, a contradiction. If $d=3$, then $r^*$ divides $\left(3(v-1),6f(q^2-1)\right)$. Since $\left(3(v-1),q^2-1\right)\mid3$, we deduce that $r^* \mid 18f$. Hence $v<18^2 f^2$, which cannot hold for any $q$ odd satisfying $(3,q+1)=3$.

\textbf{Subcase 1.3:} $n$ is even. Then $H_0 \cong \mathrm{PSO}_n^\pm(q).2 $, and so 
\begin{equation}\label{v=}
    v= d^{-1}q^{n^2/4}(q^{n/2}\pm 1)(q^3+1)(q^5+1)\cdots(q^{n-1}+1).
\end{equation}
As $q^{n/2}\pm1 > \frac{1}{2}q^{n/2}$ and $q+1<2q$, we deduce that $v>\frac{1}{4}q^{n^2/2+n/2-2}$.
It follows from Lemma \ref{deisgns_divisibility}(i) that $r^*$ divides 
\begin{equation}\label{PSO_typer^*divides}
    2df q^{n(n-2)/4}\cdot (q^{n/2}\mp1)(q^2-1)(q^4-1)\cdots (q^{n-2}-1).
\end{equation}
Suppose that $n\geq6$. Then $v$ is even and we select the following factors of $v$:
\begin{align*}
    &2,q,q+1,q^3+1, \cdots, q^{(n-2)/2}+1  \text{, \quad when } (n-2)/2 \text{ is odd},\\
    &2,q,q+1,q^3+1, \cdots, q^{(n-4)/2}+1  \text{, \quad when } (n-2)/2 \text{ is even},
\end{align*}
which are all coprime to $r^*$ as $r^* \mid (v-1)$. With the bound $q^{n/2} \mp 1 \leq 2q^{n/2}$, (\ref{PSO_typer^*divides}) simplifies to show that $r^*$ divides
\begin{align*}
   & f \cdot \frac{(q^{n/2}\mp1)(q^2-1)(q^4-1)\cdots(q^{n-2}-1)}{(q+1)(q^3+1)\cdots(q^{(n-2)/2}+1)}\leq2f\cdot q^{\frac{3n^2}{16}}   \text{, \quad when } (n-2)/2 \text{ is odd};\\
   & f \cdot \frac{(q^{n/2}\mp1)(q^2-1)(q^4-1)\cdots(q^{n-2}-1)}{(q+1)(q^3+1)\cdots(q^{(n-4)/2}+1)} \leq 2f\cdot q^{\frac{3n^2}{16}+\frac{n}{4}-\frac{1}{4}} \text{, \quad when } (n-2)/2 \text{ is even.}
\end{align*}
Combining these two cases we derive that $r^* \leq 2f\cdot q^{3n^2/16+n/4-1/4}$. Since $v<(r^*)^2$ and $16f^2\le q^3$ when $q$ is odd, we deduce $n^2<36$ and this inequality does not hold for $n\ge6$. Thus $n=4$. Now $v=q^4(q^2\pm1)(q^3+1)/(q+1,4)$. By \eqref{PSO_typer^*divides}, we have that $r^*$ divides \[2f(q+1,4)\cdot q^2(q^2\mp1)(q^2-1).\]
Observe that $v$ is even and divisible by $p$, so $r^*$  divides 
\[f\cdot (q^2\mp1)(q^2-1)/4<f\cdot q^4/2.\]
Therefore we obtain $v <f^2 q^8/4$. On the other hand, since $v>q^9/8$, it follows that $q<2f^2$, which is impossible for odd $q$.

\textbf{Case 2:} Suppose that $H$ is of type $\mathrm{Sp}_n(q)$, where $n$ is even. Following a similar strategy as in \textbf{Subcase 1.3}, we consider the following factors of $v$ when $n\ge 6$:
\begin{align*}
    &2,\, q,\, q+1,\, q^3+1,\, \ldots,\, q^{n/2-1}+1 \text{, \quad when } n/2 \text{ is even},\\
    &2,\, q,\, q+1,\, q^3+1,\, \ldots,\, q^{n/2}+1  \text{,\quad \quad when } n/2 \text{ is odd},
\end{align*}
each of which is coprime to $r^*$. This allows us to derive that $r^*  \leq f\cdot q^{\frac{3n^2}{16}+\frac{n}{2}}$. Then inequality $v < (r^*)^2$ yields $q^{n^2/2-n/2-2}<4f^2\cdot q^{3n^2/8+n}<q^{3n^2/8+n+2}$, which implies $n\leq14$. As the detailed reasoning for excluding the remaining cases is analogous to that in \textbf{Subcase 1.3}, we omit  the details here.
\end{proof}

\begin{lemma}
    Assume Hypothesis \ref{Hypothesis}. Then $H\notin\mathcal{S}$.
\end{lemma}
\begin{proof}
Let $H$ be an $\mathcal{S}$ subgroup. 
Following an analogous argument to Lemma \ref{S}, we deduce $n\leq13$. For $n \leq 12$, the possibilities for $H_0$ can be read off from the tables in \cite[Chapter 8]{bray2013maximal}. For the case $n=13$, the possible $H_0$ are listed in \cite[Table 11.0.4]{schroder2015maximal}. Combining Lemma \ref{lemma|X|<|Out|^2|H_0|....} with the inequality $|X|>q^{n^2-3}$ in \cite[Lemma 4.2]{alavi2015large}, we have
\begin{equation}\label{s_subgroup}
   q^{n^2-3} < 4d^2f^2\cdot |H_0|\cdot|H_0|^2_{p'}<4n^2f^2\cdot |H_0|^3.
\end{equation}
 Applying (\ref{s_subgroup}) and the conditions on $q$ in \cite{bray2013maximal,schroder2015maximal}, we determine   possible $q$-values, and list all  pairs $(X,H_0)$  in Table \ref{tab:S_subgroup}. 
 For each of these cases, we compute $v$ and $|\mathrm{Out}(X)|\cdot|H_0|$. Setting $R=\left(v-1,|\mathrm{Out}(X)|\cdot|H_0|\right)$,  the condition   $v<R^2$  rules out all cases except Line 1 with $q=3$ in Table \ref{tab:S_subgroup}. In this case, $v=36$ and $r^* \mid 7$. It follows from $v<(r^*)^2$ that $r^*=7$. Then we have $r=7(r,\lambda)$. Since $r$ divides $|\mathrm{Out}(X)|\cdot|H_0|=336$, we have $(r,\lambda)\mid48$. By $\lambda^*(v-1)=r^*(k-1)$, we have $k=5\lambda^*+1$ and
\[b=vr/k=252(r,\lambda)/(5\lambda^*+1)\in\mathbb{Z}.\]
Combining this with $\lambda v<r^2$ we deduce that $$(v,b,r,k,\lambda)=(36,36,21,21,12) \text{ or } (36,48,28,21,16).$$ However, Magma \cite{MAGMAbosma1994handbook}  verifies that no non-trivial flag-transitive $2$-designs exist for  these parameters.
 \end{proof}

\begin{table}[H]
\caption{\text{The pairs $(X,H_0)$}}
\label{tab:S_subgroup}
\centering
\begin{tabular}{llll}
\toprule
Line & $X$ & $H_0$ & \text{Possible } $q$   \text{values }\\
\midrule
        1  & $\mathrm{PSU}_3(q)$  & $\mathrm{PSL}_2(7)$   & $q=3,5,13,17,19$ \\
        2  &             & $\mathrm{A}_6$         &  $q=11,29$ \\
        3  &             & ${\mathrm{A}_6}^\cdot 2_3 $     & $q=5$ \\
        4  &             & $\mathrm{A}_7$                  & $q=5$ \\
\midrule
         5& $\mathrm{PSU}_4(q)$  & $\mathrm{A}_7$  & $q=3,5$\\
         6&   & $\mathrm{PSL}_3(4)$ & $q=3$\\
       7 &  & $\mathrm{PSU}_4(2)$ &    $q=5,11$\\
\midrule

         8&  $\mathrm{PSU}_5(q)$  & $\mathrm{PSL}_2(11)$     &  $q=2$ \\

\midrule
      9  &$\mathrm{PSU}_6(q)$& $\mathrm{M}_{22}$    & $q=2$ \\
      10  && $\mathrm{PSU}_4(3):2_2$ & $q=2$\\



       11 &  $\mathrm{PSU}_9(q)$ & $\mathrm{J}_3$      & $q=2$  \\
         

      


\bottomrule
\end{tabular}
\end{table}

\section{Acknowledgements}
This work was supported by the National Natural Science Foundation of China under Grant No. 12501467.



\end{document}